\documentclass[12pt]{amsart}
\usepackage{amssymb,amsmath,amsfonts,latexsym,amscd}
\usepackage{bm}
\newcommand{\newsection}[1]{\setcounter{equation}{0} \section{#1}}
\theoremstyle{plain}
\newtheorem{propn}{Proposition}[section]
\newtheorem{thm}[propn]{Theorem}
\newtheorem{lemma}[propn]{Lemma}
\newtheorem{cor}[propn]{Corollary}
\newtheorem*{thm*}{Theorem}

\theoremstyle{definition}

\newtheorem*{note}{Note}

\theoremstyle{remark}
\newtheorem*{rem}{Remark}

\newcommand{\q}{\mathcal{Q}}
\newcommand{\s}{\mathcal{S}}

\newcommand{\clb}{\mathcal{B}}

\newcommand{\cle}{\mathcal{E}}

\newcommand{\clh}{\mathcal{H}}

\newcommand{\clp}{\mathcal{P}}

\newcommand{\cls}{\mathcal{S}}

\newcommand{\z}{\bm{z}}
\newcommand{\w}{\bm{w}}

 \newcommand{\D}{\mathbb{D}}
\newcommand{\N}{\mathbb{N}}
\newcommand{\ot}{\otimes}

\newcommand{\vp}{\varphi}

\newcommand{\bt}{\begin{Theorem}}
\def\beginlem{\begin{Lemma}}
\def\beginprop{\begin{Proposition}}
\def\begincor{\begin{Corollary}}
\def\begindef{\begin{Definition}}
\def\beginexamp{\begin{Example}}
\def\beginrem{\begin{Remark}}
\def\beginq{\begin{Question}}
\def\beginass{\begin{ass}}
\def\beginnote{\begin{Note}}
\newcommand{\et}{\end{Theorem}}
\def\endlem{\end{Lemma}}
\def\endprop{\end{Proposition}}
\def\endcor{\end{Corollary}}
\def\enddef{\end{Definition}}
\def\endexamp{\end{Example}}
\def\endq{\end{Question}}
\def\endass{\end{ass}}
\def\endnote{\end{Note}}

\DeclareMathOperator{\Ran}{Ran}

\begin{document}


\title{Inner multipliers and Rudin type invariant subspaces}

\author[Chattopadhyay] {Arup Chattopadhyay}
\address{
(A. Chattopadhyay) Indian Statistical Institute \\ Statistics
andMathematics Unit \\ 8th Mile, Mysore Road \\ Bangalore \\ 560059
\\ India}

\email{2003arupchattopadhyay@gmail.com, arup@isibang.ac.in}

\author[Das] {B. Krishna Das}

\address{
(B. K. Das) Indian Statistical Institute \\ Statistics and
Mathematics Unit \\ 8th Mile, Mysore Road \\ Bangalore \\ 560059 \\
India} \email{dasb@isibang.ac.in, bata436@gmail.com}

\author[Sarkar]{Jaydeb Sarkar}

\address{
(J. Sarkar) Indian Statistical Institute \\ Statistics and
Mathematics Unit \\ 8th Mile, Mysore Road \\ Bangalore \\ 560059
\\ India}

\email{jay@isibang.ac.in, jaydeb@gmail.com}


\begin{abstract}
Let $\mathcal{E}$ be a Hilbert space and
$H^2_{\mathcal{E}}(\mathbb{D})$ be the $\cle$-valued Hardy space
over the unit disc $\mathbb{D}$ in $\mathbb{C}$. The well known
Beurling-Lax-Halmos theorem states that every shift invariant
subspace of $H^2_{\cle}(\D)$ other than $\{0\}$ has the form $\Theta
H^2_{\cle_*}(\D)$, where $\Theta$ is an operator-valued inner
multiplier in $H^\infty_{B(\cle_*, \mathcal{E})}(\mathbb{D})$ for
some Hilbert space $\cle_*$. In this paper we identify
$H^2(\mathbb{D}^n)$ with $H^2(\mathbb{D}^{n-1})$-valued Hardy space
$H^2_{H^2(\mathbb{D}^{n-1})}(\mathbb{D})$ and classify all such
inner multiplier $\Theta \in
H^\infty_{\mathcal{B}(H^2(\mathbb{D}^{n-1}))}(\mathbb{D})$ for which
$\Theta  H^2_{H^2(\mathbb{D}^{n-1})}(\mathbb{D})$ is a Rudin type
invariant subspace of $H^2(\mathbb{D}^n)$.
\end{abstract}

\maketitle

\section*{Notation}

\begin{list}{\quad}{}
\item $\cle$ \quad \quad \; Separable Hilbert space.
\item $\mathbb{N}$ \quad \quad \; Set of all natural numbers including $0$.

\item $\mathbb{N}^n$ \quad \quad $\{\bm{k} = (k_1, \ldots, k_n) : k_i \in \mathbb{N}, i = 1,\ldots, n\}$.
\item $\mathbb{C}^n$  \quad \quad Complex $n$-space.
\item $\bm{z}$ \quad \quad \; $(z_1, \ldots, z_n) \in \mathbb{C}^n$.
\item $\bm{z}^{\bm{k}}$ \quad \quad \,$z_1^{k_1}\cdots
z_n^{k_n}$.
\item $\mathbb{D}^n$  \quad \quad Open unit polydisc $\{\z : |z_i|
<1, i=1,\dots,n\}$.
\item $\mathbb{T}^n$  \quad \quad $\{\z : |z_i|
= 1, i=1,\dots,n\}$- Distinguished boundary of $\mathbb{D}^n$.
\end{list}

Throughout this article, we denote by $\mathcal{B}(\cle_*, \cle)$
the space of all bounded linear operators from $\cle_*$ to $\cle$
and simply write $\clb(\cle)$ when $\cle=\cle_*$. For a closed
subspace $\mathcal{S}$ of a Hilbert space $\clh$,  $P_{\mathcal{S}}$
denotes the orthogonal projection onto $\mathcal{S}$.

\newsection{Introduction}\label{sec:1}

The $\cle$-valued Hardy space over $\D^n$ is denoted by $H^2_{\cle}(\D^n)$ and
defined by
\[
H^2_{\cle}(\D^n):= \{ f(\z) = \sum_{\bm{k} \in \N^n} \z^{\bm{k}}
\eta_{\bm{k}}: ~~~\|f\|^2 := \sum_{\bm{k} \in \N^n}
\|\eta_{\bm{k}}\|^2_{\cle} < \infty,~~~\z\in \D^n\}.
\]
A closed subspace $\mathcal{S}\subseteq H^2_{\cle}(\D^n)$ is said to
be a \emph{shift invariant subspace}, or simply an
\textit{invariant subspace}, of  $H^2_{\cle}(\D^n)$ if $\cls$ is
invariant under the shift operators $\{M_{z_1}, \ldots, M_{z_n}\}$,
that is, if $M_{z_i}\mathcal{S}\subseteq \mathcal{S}$ for all $i =
1, \ldots, n$. Here the tuple of \textit{shift} operators $\{M_{z_1},
\ldots, M_{z_n}\}$ on $H^2_{\cle}(\D^n)$ is defined by
\[\left(M_{z_i} f\right)(\w) = w_i f(\w),
\quad \quad (f\in H^2_{\cle}(\D^n), \w\in \D^n)\]for all $i = 1,
\ldots, n$. The Banach space of all $\clb(\cle_*, \cle)$-valued
bounded analytic functions on $\D^n$ is denoted by
$H^\infty_{\clb(\cle_*, \cle)}(\D^n)$. Each $\Theta \in
H^\infty_{\clb(\cle_*, \cle)}(\D^n)$ induces a bounded linear map
$M_{\Theta} \in \clb(H^2_{\cle_*}(\D^n), H^2_{\cle}(\D^n))$ defined
by \[(M_{\Theta} f)(\w) = \Theta(\w) f(\w). \quad \quad (f \in
H^2_{\cle_*}(\D^n), \w \in \D^n)\]The elements of
$H^\infty_{\clb(\cle_*, \cle)}(\D^n)$ are called the
\textit{multipliers} and are determined by \[\Theta \in
H^\infty_{\clb(\cle_*, \cle)}(\D^n) \Leftrightarrow M_{z_i}
M_{\Theta} = M_{\Theta} M_{z_i}, \quad \quad \forall i = 1, \ldots,
n\] where the shift $M_{z_i}$ on the left hand side and the right
hand side act on $H^2_{\cle}(\D^n)$ and $H^2_{\cle_*}(\D^n)$
respectively. A multiplier $\Theta \in H^\infty_{\clb(\cle_*,
\cle)}(\D^n)$ is said to be \textit{inner} if $M_{\Theta}$ is an
isometry, or equivalently, $\Theta(\z) \in \clb(\cle_*,\cle)$ is an
isometry almost everywhere with respect to the Lebesgue measure on
$\mathbb{T}^n$.

Inner multipliers are among the most important tools for classifying
invariant subspaces of reproducing kernel Hilbert spaces. For
instance:

\begin{thm}\textsf{(Beurling-Lax-Halmos \cite{NF})}
A non-zero closed subspace $\mathcal{S}\subseteq H^2_{\cle}(\D)$ is
shift invariant if and only if there exists an inner multiplier
$\Theta \in H^\infty_{\clb(\cle_*,\cle)}(\D)$ such that \[\cls =
\Theta H^2_{\cle_*}(\D),\] for some Hilbert space $\cle_*$.
\end{thm}

For the Hardy space $H^2(\D^n)$, $n \geq 2$, Beurling-Lax-Halmos
theorem and most of its corollaries turns out to be false in
general (see Rudin \cite{rudin}). In fact, it is shown in
~\cite{SSW} that Beurling-Lax-Halmos
theorem holds for an invariant subspace of $H^2(\D^n)$
 if and only if it is doubly commutating. Recall that
a closed shift-invariant subspace $\cls \subseteq H^2(\mathbb{D}^n)$
is said to be \textit{doubly commuting} if
\[
R_{z_i}R_{z_j}^* = R_{z_j}^*R_{z_i}, \quad \quad (1\leq i\neq j\leq n)
\]where
\[
R_{z_i} = M_{z_i}|_{\mathcal{S}}. \quad \quad (i = 1, \ldots, n)
\]
\begin{thm}[\cite{SSW}]
\label{ssw}
Let $\s \neq \{0\}$ be a closed shift-invariant subspace of $H^2(\D^n)$, $n \geq 2$. Then the following are equivalent.
\begin{itemize}
\item[(i)] $\s$ is a doubly commuting shift-invariant subspace.
\item[(ii)] $\s=\vp H^2(\D^n)$ for some inner function $\vp$ in $H^2(\D^n)$.
\end{itemize}
\end{thm}

The analytic structure of invariant subspaces of $H^2(\D^n)$, $n
\geq 2$, is more complicated than that of the Hardy space $H^2(\D)$
(see \cite{ADC}, \cite{AC}, \cite{DS}, \cite{rudin}, \cite{RYang1},
\cite{RYang2}, \cite{RYang3}).

Now let $n \geq 2$ and $\Theta \in
H^\infty_{\clb(H^2(\D^{n-1}))}(\D)$ be an inner multiplier. Then
$\Theta H^2_{H^2(\D^{n-1})}(\D) \subseteq H^2_{ H^2(\D^{n-1})}(\D)$
and hence by identifying $H^2_{H^2(\D^{n-1})}(\D)$ with $H^2(\D^n)$,
that $\Theta H^2_{H^2(\D^{n-1})}(\D)$ is a closed
$M_{z_1}$-invariant subspace of $H^2(\D^n)$.

Thus, it is natural ask to what extent the
structure of inner multipliers determines the structure of invariant
subspaces. That is, how to determine inner multiplier $\Theta \in
H^\infty_{\clb(H^2(\D^{n-1}))}(\D)$ such that $\Theta
H^2_{H^2(\D^{n-1})}(\D)$ is an invariant subspace of $H^2(\D^n)$?

The purpose of this paper is to study the above problem for a
special class of inner multipliers (see the definition (\ref{theta})
in the next section) and to provide a general recipe for producing
invariant subspaces of $H^2(\D^n)$. More precisely, our purpose here
is to deduce more detailed structure of invariant subspaces of
$H^2(\D^n)$, $n \geq 2$, from Beurling-Lax-Halmos inner multipliers.
We refer to \cite{QY1} and \cite{JS1} for some closely related
constructions of inner multipliers.

The approach that we will take is inspired by the recent work of Y.
Yang \cite{Yang1}. However, our results
improve and generalize many results proved for the base case $n = 2$
in \cite{Yang1}.

The paper is organized as follows. In section 2 we introduce some
notations and definitions. Our main results are in Section 3. The
last section of the paper, Section 4, is devoted to the study of the
unitarily equivalent invariant subspaces of $H^2(\D^n)$.

\newsection{Notations and definitions}

We will often identify $H^2(\D^n)$ with the $n$-fold Hilbert space
tensor product $H^2(\D) \otimes \cdots \otimes H^2(\D)$ via the
unitary map $H^2(\D) \otimes \cdots \otimes H^2(\D) \ni z^{k_1}
\otimes \cdots \otimes z^{k_n} \mapsto z^{\bm{k}} \in H^2(\D^n)$,
$\bm{k} \in \N^n$. Therefore we can, and do, identify $M_{z_i}$ with
\[I_{H^2(\D)} \otimes \cdots \otimes
\underbrace{M_z}\limits_{\textup{i-th place}} \otimes \cdots \otimes
I_{H^2(\D)}. \quad \quad (i=1,\dots,n)\]

A sequence of inner functions $\{\varphi_j\}_{j=1}^{\infty}$ in
$H^\infty(\D^n)$ is said to be increasing (respectively, decreasing)
if $\frac{\varphi_{j+1}}{\varphi_{j}}$ (respectively,
$\frac{\varphi_j}{\varphi_{j+1}}$) is a non-constant inner function
for every $j\geq 1$. An inner sequence is a sequence of inner
functions $\{\varphi_j\}_{j=1}^{\infty}$ in $H^\infty(\D^n)$ which
is either increasing or decreasing.

A sequence of pairwise orthogonal projections $\{P_j\}_{j=1}^\infty$
on $H^2(\mathbb{D}^n)$ is said to be a \emph{sequence of orthogonal
complementary projections} if \[\sum \limits_{j=1}^{\infty}P_j =
I_{H^2(\D^{n})},\]in strong operator topology. The set of sequences
of orthogonal complementary projections on $H^2(\D^n)$ will be
denoted by $\clp_n$.

\textsf{From now on, we will assume that $n \geq 2$.}

Let $\{P_j\}_{j=1}^\infty \in \clp_{n-1}$ and
$\{\varphi_j\}_{j=1}^{\infty} \subseteq H^\infty(\D)$ be an inner
sequence. Then
\begin{equation}
\Theta(z) = \sum_{j=1}^{\infty} \varphi_j(z) P_j, \quad \quad (z \in
\D) \label{theta}
\end{equation}
is a $\mathcal{B}(H^2(\D^{n-1}))$-valued analytic function on $\D$.

The following lemma is an immediate consequence of the definition.

\begin{lemma}
Let $\Theta$ be as in ~\eqref{theta}. Then
$\Theta \in H^\infty_{\mathcal{B}(H^2(\D^{n-1}))}(\D)$ is an inner
multiplier.
\end{lemma}

The primary goal of this paper is to present a complete
characterization of inner multipliers, defined above (depending on
$\{P_j\}_{j=1}^\infty \in \clp_{n-1}$ and inner sequence
$\{\varphi_j\}_{j=1}^\infty$), for which the corresponding closed
subspaces in $H^2(\D^n)$ are shift invariant. Our approach is also
related to the study of Rudin type invariant subspaces of $H^2(\D^n)$.

An invariant subspace $\cls$ of $H^2(\D^n)$ is said to be of Rudin
type if there exists an integer $1\le k<n$, an increasing sequence of
inner functions $\{\vp_j\}_{j=1}^{\infty} \subseteq
H^\infty(\D^{k})$ and a decreasing sequence of inner functions
 $\{\psi_j\}_{j=1}^{\infty} \subseteq H^\infty (\D^{n-k})$ such that
\[
\cls = \bigvee_{j=1}^{\infty}\vp_jH^2(\D^{k})\ot
\psi_jH^2(\D^{n-k}).
\]

These invariant subspaces, also known as inner sequence based
invariant subspaces of $H^2(\D^n)$, have been studied extensively by
various authors in different contexts (see \cite{CDS}, \cite{DS},
\cite{I2}, \cite{rudin}, \cite{MS}, \cite{Se}).

\newsection{Inner multipliers and invariant subspaces}\label{sec:2}

In this section, we will prove the main result concerning inner
multipliers based shift invariant subspaces of $H^2(\D^n)$.
To begin with, we prove a result concerning invariant subspaces
corresponding to a sequence of orthogonal complementary projections
in $H^2(\D^n)$, which will be used to establish our main result.

\begin{lemma}\label{doubly commuting}
Let $\{P_j\}_{j=1}^{\infty} \in \clp_n$ and $\s_k:=
\bigoplus_{j=k}^{\infty} \Ran P_j$ be an invariant subspace of
$H^2(\D^n)$ for each $k\geq 1$. Then the following are equivalent
\begin{itemize}
\item[(i)]
$\s_k$ is doubly commuting for all $k\geq 1$.
\item[(ii)]
For all $1 \leq p \neq q \leq n$ and $j\ge 1$,
\[
P_{\s_l}M_{z_{p}}P_jM_{z_{q}}^*P_{\s_m} = 0.\quad (l,m\ge j+1)
\]
\item[(iii)]
For all $1\leq p \neq q \leq n$ and $j\ge 1$,
\[P_lM_{z_{p}}P_jM_{z_{q}}^*P_m = 0.\quad (l,m\geq j+1)
\]
\end{itemize}
\end{lemma}

\begin{proof}
It is easy to see that (ii) $\Leftrightarrow$ (iii). Therefore, it
is enough to prove that (i) $\Leftrightarrow$ (ii).

\noindent We first prove that (i) implies (ii). Let $\s_k$, $k\geq
1$, be a doubly commuting subspace . By Theorem~\ref{ssw}, there is
an increasing sequence of inner functions
$\{\varphi_j\}_{j=1}^\infty \subseteq H^\infty(\D^n)$ such that
$\s_k = \vp_k H^2(\D^n)$, $k \geq 1$. Then
\[\Ran P_{j}={\varphi_j} H^2(\D^n) \ominus {\varphi_{j+1}}H^2(\D^n),
\]and
\[P_j = M_{\vp_j} M_{\vp_j}^* - M_{\vp_{j+1}} M_{\vp_{j+1}}^*
= M_{\vp_j} (I - M_{\xi_j} M_{\xi_j}^*) M_{\vp_j}^*,\]where
$\varphi_{j+1} = \xi_j \varphi_{j}$ for some inner function $\xi_j
\in H^\infty(\D^n)$, $j\ge 1$. Consequently for each $j \geq 1$ and
$1 \leq p< q \leq n$, we have
\[\begin{split}P_{\cls_{\vp_{j+1}}} M_{z_p} P_j M_{z_q}^* P_{\cls_{\vp_{j+1}}} &
= M_{\vp_{j+1}} (M_{\vp_{j+1}}^* M_{\vp_{j}} M_{z_p} (I - M_{\xi_j}
M_{\xi_j}^*) M_{z_q}^* M_{\vp_{j}}^* M_{\vp_{j+1}}) M_{\vp_{j+1}}^*\\
& = M_{\vp_{j+1}} (M_{\xi_{j}}^* M_{z_p} (I - M_{\xi_j} M_{\xi_j}^*)
M_{z_q}^* M_{\xi_{j}} ) M_{\vp_{j+1}}^*\\
& = M_{\vp_{j+1}} M_{\xi_{j}}^* (M_{z_q}^* M_{z_p} - M_{\xi_j}
M_{z_p} M_{z_q}^* M_{\xi_j}^*)M_{\xi_{j}} M_{\vp_{j+1}}^*\\
& = M_{\vp_{j+1}} (M_{\xi_{j}}^* M_{z_q}^* M_{z_p} M_{\xi_{j}}  -
M_{z_p} M_{z_q}^* ) M_{\vp_{j+1}}^*\\ & = 0.
\end{split}\]
Finally, by multiplying the above on the left and right by
$P_{\s_l}$ and $P_{\s_m}$ ($l, m > j$), respectively, we get the
desired equality.

We now prove that (ii) implies (i). Let $1\leq p < q \leq n$ and
$k\geq 1$. Then
\[
\begin{split}
P_{\s_k} M_{z_{q}}^* M_{z_{p}}|_{\s_k} & = P_{\s_k} M_{z_{p}}
M_{z_{q}}^*|_{\s_k} = P_{\s_k} M_{z_{p}}\left(P_{\s_k}+P_{\s_k^{\perp}}\right) M_{z_{q}}^*|_{\s_k}\\
& = P_{\s_k} M_{z_{p}} P_{\s_k} M_{z_{q}}^*|_{\s_k} +
P_{\s_k} M_{z_{p}}\left(\sum_{j=1}^{k-1}P_j\right) M_{z_{q}}^*|_{\s_k}\\
& =P_{\s_k} M_{z_{p}} P_{\s_k} M_{z_{q}}^*|_{\s_k}  +
\sum_{j=1}^{k-1}P_{\cls_k} M_{z_{p}}  P_j M_{z_{q}}^* P_{\cls_k}\\
& = P_{\s_k} M_{z_{p}} P_{\s_k} M_{z_{q}}^*|_{\s_k}, \quad \quad
(\mbox{by~} (ii))
\end{split}
\]
that is, $ (M_{z_{q}}|_{\s_k})^*
(M_{z_{p}}|_{\s_k})=(M_{z_{p}}|_{\s_k}) (M_{z_{q}}|_{\s_k})^*$, or
equivalently, $\s_k$ is doubly commuting for all $k\ge 1$. This
completes the proof.
\end{proof}

We are now ready to state and prove our main result.

\begin{thm}\label{thm2}
Let $\{P_j\}_{j=1}^{\infty} \in \clp_{n-1}$ and
$\{\psi_j\}_{j=1}^{\infty} \subseteq H^\infty(\D)$ be a decreasing inner
sequence. Set $\Theta = \sum\limits_{j=1}^{\infty} \psi_jP_j$, $\s
=\Theta H^2_{H^2(\D^{n-1})}(\D)$ and $\s_j:=
\bigoplus_{k=j}^{\infty}\Ran P_k,
j\ge 1$. \\
\textup{(a)} $\s$ is an invariant subspace of $H^2(\D^n)$ if and
only if $\s_j$ is an invariant subspace
of $H^2(\D^{n-1})$ for all $j\ge 1$.\\
\textup{(b)} The following are equivalent
\begin{itemize}
\item[(i)] There exists an increasing inner sequence $\{\varphi_j\}_{j=1}^\infty \subseteq H^\infty(\D^{n-1})$ such that
\[
\s=\bigvee_{j=1}^{\infty}\psi_jH^2(\D)\ot \varphi_jH^2(\D^{n-1}).
\]
\item[(ii)]
$\s_j$ is a doubly commuting invariant subspace of $H^2(\D^{n-1})$
for all $j\ge 1$.
\item[(iii)]
For each $j \geq 1$, $\s_j$ is an invariant subspace of
$H^2(\D^{n-1})$ and
\[
P_{\s_l}M_{z_{p}}P_j M_{z_{q}}^*P_{\s_m} = 0.\quad (l,m > j,\ 1 \leq p <q \leq n-1)
\]
\item[(iv)]
For each $j \geq 1$, $\s_j$ is an invariant subspace of
$H^2(\D^{n-1})$ and
\[
P_l M_{z_{p}}P_j M_{z_{q}}^* P_m = 0.\quad (l,m > j,\ 1 \leq p <q \leq n-1)
\]
\end{itemize}
\end{thm}

\begin{proof}
Set $\psi_0:=0$. First note that
$\cls=\bigoplus_{j=1}^{\infty}\psi_j H^2(\D)\otimes \Ran P_j$. Since
the inner sequence $\{\psi_j\}_{j=1}^{\infty}$
is decreasing, we have $\psi_jH^2(\D)\subset \psi_{j+1}H^2(\D)$ for all
$j\ge 1$, and
\begin{align*} \bigoplus_{j=1}^{\infty}\Big((\psi_j H^2(\D) \ominus \psi_{j-1}
H^2(\D))\otimes\s_j\Big)& = \bigoplus_{j=1}^{\infty}\Big((\psi_j
H^2(\D) \ominus \psi_{j-1} H^2(\D))\otimes
\big(\bigoplus_{k=j}^{\infty}\Ran P_k\big)\Big)\\
&= \bigoplus_{j=1}^{\infty}\Big(\bigoplus_{k=1}^{j}(\psi_k H^2(\D) \ominus \psi_{k-1}
H^2(\D))\otimes \Ran P_j\Big)\\
&=\bigoplus_{j=1}^{\infty}\Big(\psi_jH^2(\D)\otimes \Ran P_j\Big),
\end{align*}
where for the last equality we use
$ \bigoplus_{k=1}^{j}(\psi_kH^2(\D) \ominus \psi_{k-1} H^2(\D))=\psi_jH^2(\D)\quad
 (j\ge 1).$
Thus
\begin{equation}
\label{s-psi}
\cls = \bigoplus_{j=1}^{\infty}\Big((\psi_j H^2(\D) \ominus \psi_{j-1}
H^2(\D))\otimes\s_j\Big).
\end{equation}


\noindent Proof of part (a): Let $\s$ be an invariant subspace of
$H^2(\D^n)$ and $j \geq 1$ be a fixed integer. Let $f\in\psi_j
H^2(\D) \ominus \psi_{j-1} H^2(\D)$, $g\in \s_j$ and $2\le i\le n$.
Since
\[z_i(f \otimes g) = f \otimes z_i g \in \cls =\bigoplus_{k=1}^{\infty}\left(\psi_k H^2(\D) \ominus
\psi_{k-1} H^2(\D)\right)\otimes \s_k,\]and $f \otimes z_i g \perp
(\psi_k H^2(\D) \ominus \psi_{k-1} H^2(\D))\otimes \s_k$ for all $k
\neq j$, we have $f \otimes z_i g \in (\psi_j H^2(\D) \ominus
\psi_{j-1} H^2(\D)) \otimes \cls_j$ and hence $z_i g \in \cls_j$.

\noindent Conversely, let $\cls_j$ be an invariant subspace of
$H^2(\D^{n-1})$ for all $j \geq 1$. Then by (\ref{s-psi}) it follows
that $\cls$ is joint $\{M_{z_2}, \cdots, M_{z_n}\}$-invariant.
Finally, since $\cls = \Theta H^2_{H^2(\D^{n-1})}(\D)$, it follows
that $\cls$ is $M_{z_1}$-invariant.

\noindent Proof of part (b): (ii)$\Leftrightarrow$ (iii)
$\Leftrightarrow$ (iv) follows from Lemma \ref{doubly commuting}.
Now assume that (i) is true. Then
\[
\s=\bigvee_{j=1}^{\infty}\psi_jH^2(\D)\ot \varphi_jH^2(\D^{n-1})=
\bigoplus_{j=1}^\infty\left(\psi_j H^2(\D) \ominus \psi_{j-1}
H^2(\D)\right)\ot \varphi_jH^2(\D^{n-1}).
\]
Comparing this with ~\eqref{s-psi}, we have
$\s_j=\varphi_jH^2(\D^{n-1})$ for all $j\ge 1$. Then by
Theorem~\ref{ssw}, $\s_j$ is doubly commuting for all $j\ge 1$.

\noindent Conversely assume (ii). Then by Theorem~\ref{ssw}, there
exists a sequence of increasing inner functions
$\{\varphi_j\}_{j=1}^\infty \subseteq H^2(\D^{n-1})$ such that
$\s_j=\varphi_jH^2(\D^n)$, $j\ge 1$. Then (i) follows from
\eqref{s-psi}. This completes the proof.
\end{proof}

One can reformulate the above theorem by replacing the decreasing
inner sequence by an increasing one.

\begin{thm}
Let $\{P_j\}_{j=1}^{\infty} \in \clp_{n-1}$ and
$\{\varphi_j\}_{j=1}^{\infty} \subseteq H^\infty(\D)$ be an increasing
inner sequence.

Set $\Theta = \sum\limits_{j=1}^{\infty} \phi_jP_j$, $\s =\Theta
H^2_{H^2(\D^{n-1})}(\D)$ and $\s_j:=\bigoplus_{k=1}^{j}\Ran P_k,
j\ge 1$. \\
\textup{(a)} $\s$ is an invariant subspace of $H^2(\D^n)$ if and
only if $\s_j$ is an invariant subspace
of $H^2(\D^{n-1})$ for all $j\ge 1$.\\
\textup{(b)} There exists a decreasing inner sequence
$\{\psi_j\}_{j=1}^\infty \subseteq H^\infty(\D^{n-1})$ such that
\[
\s=\bigvee_{j=1}^{\infty}\phi_jH^2(\D)\ot \psi_jH^2(\D^{n-1}),
\]
if and only if $\s_j$ is a doubly commuting invariant subspace of
$H^2(\D^{n-1})$ for all $j\ge 1$.
\end{thm}

\begin{proof}
(a) We first note that, under the given assumptions, the subspace
$\s$ is given by
\[
\s=\bigoplus_{j=1}^\infty\varphi_j(z_1)H^2(\D)\ot \Ran P_j=
\bigoplus_{j=1}^{\infty}\big(\varphi_j(z_1)H^2(\D)\ominus
\varphi_{j+1}(z_1)H^2(\D)\big) \ot \s_j.
\]By the same argument as in part (a) of Theorem \ref{thm2}, it follows that $\s$
is an invariant subspace of $H^2(\D^n)$ if and only if $\s_j$ is an
invariant subspace of $H^2(\D^{n-1})$ for all $j\ge 1$.

\noindent (b) The proof is identical to the proof of part (b) in
Theorem~\ref{thm2} except the fact that one obtains a decreasing
inner sequence corresponding to the increasing doubly commuting
invariant subspaces $\{\cls_j\}_{j=1}^\infty$ of $H^2(\D^{n-1})$.
\end{proof}

\begin{rem}
A modification of our argument yields a similar characterization of
invariant subspaces of $H^2(\D^n)$ corresponding to the inner
multiplier
$\Theta(z_1,\dots,z_k)=\sum_{j=1}^\infty\varphi_j(z_1,\dots,z_k)P_j$,
where $\{\varphi_j\}_{j=1}^\infty$ is a decreasing or increasing
inner sequence in $H^\infty(\D^k)$ and $\{P_j\}_{j=1}^\infty \in
\clp_{n-k}$.
\end{rem}

\newsection{Unitarily equivalent invariant subspaces}

Let $\s_1$ and $\s_2$ be two invariant subspaces of $H^2(\D^n)$.
Then $\s_1$ and $\s_2$ are said to be unitarily equivalent if there
exists a unitary operator $U:\s_1 \longrightarrow \s_2$ such that
\[
UM_{z_i}|_{\s_1} = M_{z_i}|_{\s_2}U. \quad \text{for} \quad
(i=1,\ldots,n)
\]
The unitary equivalence of inner sequence based invariant subspaces
and two inner sequences based invariant subspaces of $H^2(\D^2)$ are
completely described in ~\cite{MS} and ~\cite{Yang1}, respectively.
Here we present a similar result for Rudin type invariant subspaces
in $n$-variables. The proof follows along the same lines as in
Theorem 3.1 in \cite{Yang1}.

\begin{thm}
Let $\{\vp_j\}_{j=1}^{\infty}, \{\widetilde{\vp}_j\}_{j=1}^{\infty}
\subseteq H^\infty(\D)$ be two decreasing inner sequences and
$\{\psi_j\}_{j=1}^{\infty},$ $\{\widetilde{\psi}_j\}_{j=1}^{\infty}
\subseteq H^\infty(\D^{n-1})$ be two increasing inner sequences with
$\psi_1=1=\widetilde{\psi}_1$. Let
\[
\s= \bigvee_{j=1}^{\infty} \vp_j H^2(\D) \otimes
\psi_jH^2(\D^{n-1}), \quad \mbox{and} \quad \widetilde{\s}=
\bigvee_{j=1}^{\infty} \widetilde{\vp}_j H^2(\D) \otimes
\widetilde{\psi}_jH^2(\D^{n-1}).
\]
Then $\s$ and $\widetilde{\s}$ are unitarily equivalent if and only
if there exists an inner function $\eta\in H^\infty(\D^{n})$,
depending only on the first variable $z_1$, such that $\s=
\eta\widetilde{\s}$.
\end{thm}
\begin{proof}
It is enough to prove the necessary part. Let $\s$ and
$\widetilde{\s}$ be unitarily equivalent. Then $\s$ $=\eta
\widetilde{\s}$ for some unimodular function $\eta \in
L^\infty(\mathbb{T}^n)$ (see Lemma 1 in \cite{ADC}). Then both
$\eta\widetilde{\varphi}_1(z_1)$ and $\overline{\eta}\varphi_1(z_1)$ are
in $H^2(\D^n)$, and therefore $\eta$ is holomorphic and
anti-holomorphic in $z_2,\ldots,z_{n}$. Thus $\eta$ depends only on
$z_1$ variable. This completes the proof.
\end{proof}

For more results related to unitarily equivalent invariant subspaces
of $H^2(\D^n)$, $n \geq 2$, we refer the readers to \cite{ADC},
\cite{JS1} and \cite{RYang4}.

\vskip10pt

\noindent\textbf{Acknowledgement:} The first two authors acknowledge
with thanks financial support from the Department of Atomic Energy,
India through N.B.H.M (Grant no. 2/40(41)/2012/R\&D-II/1286 \&
Grant no. No. 2/40(25)/2014/R\&D-II/16013) Post Doctoral Fellowship, and also
grateful to Indian Statistical Institute, Bangalore Center for
warm hospitality.


\end{document}